\documentclass[a4paper,10pt]{amsart}
\usepackage{amsmath,amsthm,amssymb,amsfonts,enumerate,color}

\newcommand{\Real}{\mathbb{R}}
\renewcommand{\H}{\mathcal{H}}
\renewcommand{\O}{\mathcal{O}}

\newtheorem{thm}{Theorem}[section]

\newtheorem{lem}[thm]{Lemma}

\theoremstyle{definition}
\newtheorem{defn}[thm]{Definition}
\newtheorem{rem}[thm]{Remark}

\numberwithin{equation}{section}

\author[I. Abu-Falahah]{Ibraheem Abu-Falahah}
\address{Department of Mathematics \\
	  The Hashemite University \\
	  Zarqa, Jordan}
\email{iabufalahah@hu.edu.jo}

\author[P. R. Stinga]{Pablo Ra\'ul Stinga}
\address{Departament of Mathematics\\
          The University of Texas at Austin\\
          1 University Station, C1200\\
          Austin, TX 78712-1202\\
          United States of America}
\email{stinga@math.utexas.edu}

\author[J. L. Torrea]{Jos\'e L. Torrea}
\address{Departamento de Matem\'aticas and ICMAT-CSIC-UAM-UCM-UC3M\\
          Facultad de Ciencias \\
          Universidad Au\-t\'o\-no\-ma de Madrid \\
          28049 Madrid, Spain}
\email{joseluis.torrea@uam.es}

\thanks{Research partially supported by grant MTM2011-28149-C02-01 from Gobierno de Espa\~na.}



\begin{document}

\title[Almost everywhere convergence]{A note on the almost everywhere convergence \\ to initial data for some evolution equations}

\begin{abstract}
The weighted Lebesgue spaces of initial data for which almost everywhere convergence of the heat equation holds was only very recently characterized. In this note we show that the same weighted space of initial data is optimal for the heat--diffusion parabolic equations involving the harmonic oscillator and the Ornstein--Uhlenbeck operator.
\end{abstract}

\maketitle

\section{Introduction}

Let $\mathcal{L}$ be a positive second order partial differential operator on $\mathbb{R}^n$. Consider the following parabolic problem in the upper half plane
$$\left\{
  \begin{array}{ll}
    u_t(x,t) + \mathcal{L}_x u(x,t)=0, & x\in\Real^n,~0<t<T, \\
    u(x,0)=f(x), & x\in\Real^n.
  \end{array}
\right.$$
For a very general class of operators $\mathcal{L}$ it is well known that, under mild size conditions on the initial data $f$, for example $f \in L^p(\mathbb{R}^n,dx)$,   $1\le p < \infty$,  the solution $u(x,t)$ exists and the following limit property holds:
\begin{equation}\label{pesadez}
\lim_{t\rightarrow 0^+} u(x,t)=f(x),  \qquad \hbox{for almost every}~x.
\end{equation}

Consider next the following natural question.
\begin{quote}
    \textit{Given an operator $\mathcal{L}$, for which weights $v$ the solution $u$ exists up to a certain time $T>0$ and we still have $u(x,t)\to f(x)$ almost everywhere as $t\to0^+$, for every function $f \in L^p(\mathbb{R}^n,v(x)dx)$?  }
\end{quote}
In the case of the heat equation $u_t=\Delta u$, where the solution is given by
\begin{equation}
\label{kernelclasico}u(x,t)\equiv W_tf(x)=\int_{\Real^n}W_t(x-y)f(y)\,dy,\quad x\in\Real^n,~0<t<T,
\end{equation}
with $W_t(x)=(4\pi t)^{-n/2}e^{-|x|^2/(4t)}$, the answer was given very recently by S. Hartzstein, J. L. Torrea and B. Viviani in \cite{Hartzstein-Torrea-Viviani}. They found a class  $D_p^W$ of weights $v$ such that
\begin{equation}\label{W}
v\in D^W_p\quad\hbox{if and only if}\quad
\left\{
  \begin{array}{ll}
    \displaystyle u(x,t)~\hbox{exists for a time interval}~0<t<T,~\hbox{and} \\
    \displaystyle\lim_{t\to0^+}u(x,t)=f(x)~\hbox{a.e.}~x,~\hbox{for all}~f\in L^p(\Real^n,v(x)dx).
  \end{array}
\right.
\end{equation}
The class $ D_p^W$ is described in   the following   \begin{defn}[See \cite{Hartzstein-Torrea-Viviani}]\label{Def:pesos}
Let $1\leq p<\infty$. A weight $v$ (a strictly positive, locally integrable function on $\Real^n$) belongs to the class $D_p^W$ if there exists $t_0>0$ such that
$$\|W_{t_0}v^{-1/p}\|_{L^{p'}(\Real^n,dx)}<\infty.$$
\end{defn}

In this note we show that, fixed the class $D_p^W$, we can replace the solution to the heat equation $W_tf$ in the equivalence \eqref{W} above, by the solution of the evolution equations associated with the following operators:
\begin{itemize}
\item[a)] the harmonic oscillator
    $$u_t=\Delta u-|x|^2u,$$
\item[b)] the Ornstein--Uhlenbeck operator
    $$u_t=\Delta u-2x\cdot\nabla u,$$ and
\item[c)] the strongly elliptic time-dependent parabolic operator
    \begin{equation}\label{aronson}
    u_t=\partial_j(a^{ij}(x,t)\partial_iu)+b^i(x,t)\partial_iu+c(x,t)u,
    \end{equation}
    where the coefficients are bounded and measurable functions for $(x,t)\in\Real^n\times(0,T)$.
\end{itemize}

To study the almost everywhere pointwise convergence of $u$ to $f$, $f\in L^p(\Real^n)$, as $t\to0^+$, we may consider the maximal operator $\sup_{t<T}|u(x,t)|$, for some $0<T<\infty$. It turns out that for the heat equation such operator is bounded from $L^p(\Real^n)$ into weak-$L^p(\Real^n)$ and, as a consequence, \eqref{pesadez} is true (see, for example, the book by J. Duoandikoetxea \cite[Chapter~2]{Duo} and E. Damek et al. \cite{Damek}). In fact, \eqref{W} for the classical Laplacian $\mathcal{L}_x=-\Delta$ is contained in the following Theorem.

\begin{thm}[See {\cite[Theorem~2.3]{Hartzstein-Torrea-Viviani}}]\label{Thm:clasico}
Let $1\le p< \infty$. Let $v$ be a weight in $\mathbb{R}^n$. For a locally integrable
function $f$ on $\mathbb{R}^n$ we define
$$W^*_R f(x) = \sup_{t<R}|W_tf(x)|,\quad\hbox{for some}~R,~0<R<\infty.$$
The following statements are equivalent:
\begin{enumerate}[(a)]
    \item There exists  $0< R < \infty$ and a weight $w$ such that the operator
        $$f \longmapsto W^*_R f$$
        is bounded from $L^p(\mathbb{R}^n, v(x)dx)$ into $L^p(\mathbb{R}^n,w(x)dx)$ for $p>1$. In the case $p=1$, from $L^1(\mathbb{R}^n, v(x)dx)$ into weak-$L^1(\mathbb{R}^n, w(x)dx)$.
    \item There exists $0< R < \infty$ and a weight $w$ such that the operator
        $$f \longmapsto W^*_R f$$
        is bounded from $L^p(\mathbb{R}^n, v(x)dx)$  into weak-$L^p(\mathbb{R}^n, w(x)dx)$.
    \item There exists $0<R<\infty$ such that $W_Rf(x)<\infty $ for almost every $x\in\Real^n$ and the  limit
        $$\lim_{t \rightarrow 0^+} W_tf(x)$$
        exists a.e. $x\in\Real^n$, for all $f \in L^p(\mathbb{R}^n, v(x)dx)$.
    \item There exists  $0< R < \infty$ such that
        $$W^*_R f(x) < \infty,$$
        for almost every $x\in\Real^n$, for all $ f\in L^p(\mathbb{R}^n, v(x)dx)$.
    \item The weight
        $$v \in D^W_{p}.$$
\end{enumerate}
\end{thm}

Let us consider the parabolic problem \eqref{aronson} in c) above. Assume that the coefficients are bounded and measurable functions in $\Real^n\times(0,T)$ and that $a^{ij}$ is strongly elliptic. It was proved by D. Aronson in \cite{Aronson} that the fundamental solution of \eqref{aronson} satisfies upper and lower Gaussian bounds. With this it is easily seen that we can replace $W_tf$ in Theorem \ref{Thm:clasico} above by the solution of \eqref{aronson}, a fairly general parabolic equation with \textbf{bounded} coefficients. However, in the other cases a) and b) that we present, the potential $|x|^2$ and the coefficient in the drift part $2x$ are not bounded functions; in contrast, the coefficients in Aronson's result depend on $t$.

On the other hand, a naive analysis for small times of the heat kernel for the harmonic oscillator suggests that our results could be expected. Certainly there is an \textbf{upper} Gaussian bound, but the converse is not true. Here we find a very weak estimate from below that is enough for our purposes, see Lemma \ref{Lem:dale que va} in Section \ref{Section:harmonic oscillator}.

We close this circle of ideas by pointing out that the case of the Ornstein--Uhlenbeck operator may look, in a first glance, more involved due to the appearance of the underlying Gaussian measure (some partial  results were obtained by Harboure et al.  in \cite{Harboure}). We overcome this kind of difficulties by using the transference ideas of Abu-Falahah and Torrea from \cite{Abu-Falahah-Torrea}. The transference technique of \cite{Abu-Falahah-Torrea} was also shown to be useful for transferring Harnack's inequalities among solutions of fractional nonlocal equations, see \cite{Stinga-Zhang}. Moreover, we want to emphasize that the same class of weights $D_p^W$, where no Gaussian measure is involved, characterizes the almost everywhere convergence to initial data. Again our results in this case could be expected, since the underlying stochastic diffusion process (the Ornstein--Uhlenbeck process) is the classical Wiener process subject to friction. Hence one
  would think that for small times both processes behave in a similar manner. For more information about the Ornstein--Uhlenbeck 
process see P.-A. Meyer \cite{Meyer}.

A bit more general operators could be considered, like $-\Delta+|Bx|^2$ or $-\Delta+2Bx\cdot\nabla$, for some $n\times n$ positive definite symmetric constant matrix $B$. To keep a clean presentation and to avoid rather cumbersome computations we just take $B$ to be the identity matrix.

\section{The harmonic oscillator diffusion equation}\label{Section:harmonic oscillator}

Let $\H:=-\Delta+|x|^2$ be the harmonic oscillator in $\Real^n$. We denote by $W_t^{\H}f(x)\equiv e^{-t\H}f(x)$
the solution to the initial value problem
$$\left\{
  \begin{array}{ll}
    u_t=-\H u, & \hbox{in}~\Real^n,~t>0, \\
    u(x,0)=f(x), & \hbox{on}~\Real^n.
  \end{array}
\right.$$
We are also going to consider the operator $\H-n=-\Delta+|x|^2-n$ and the corresponding solution $W^{(\H-n)}_tf(x)$ to the evolution equation. It is clear that $W_t^\H f(x)=e^{-tn}W_t^{(\H-n)}f(x)$.

\begin{thm}\label{Thm:funciones}
Let $v$ be a  weight in $\mathbb{R}^n$ and $1\le p< \infty$. Given $0<R<\infty$, consider the operators
$$T^*_R f(x) = \sup_{t<R}|W^\H_tf(x)|,\qquad\tilde{T}^*_R f(x) = \sup_{t<R}|W^{(\H-n)}_tf(x)|,$$
for $f\in L^1_{\mathrm{loc}}(\Real^n)$.
The following statements are equivalent:
\begin{enumerate}[(i)]
    \item There exists  $0< R < \infty$ and a weight $w$ such that the operator
        $$f \longmapsto \tilde{T}^*_R f$$
        is bounded from $L^p(\mathbb{R}^n, v(x)dx)$ into $L^p(\mathbb{R}^n, w(x)dx)$ for $p>1$. In the case $p=1$, from $L^1(\mathbb{R}^n,v(x)dx)$ into weak-$L^1(\mathbb{R}^n,w(x)dx)$.
    \item There exists $0< R < \infty$ and a weight $w$ such that the operator
        $$f \longmapsto \tilde{T}^*_R f$$
        is bounded from $L^p(\mathbb{R}^n, v(x)dx)$  into weak-$L^p(\mathbb{R}^n, w(x)dx)$.
    \item There exists $0<R<\infty$ such that $W^{(\H-n)}_Rf(x)<\infty $ for almost every $x\in\Real^n$ and the  limit
        $$\lim_{t \rightarrow 0^+} W^{(\H-n)}_tf(x)$$
        exists a.e. $x\in\Real^n$, for all $f \in L^p(\mathbb{R}^n, v(x)dx)$.
    \item There exists  $0< R < \infty$ such that
        $$\tilde{T}^*_R f(x) < \infty,$$
        a.e. $x$, for all $ f\in L^p(\mathbb{R}^n, v(x)dx)$.
    \item In any of the statements above the operator $\tilde{T}^\ast_R$ can be replaced by $T^\ast_R$.
    \item The weight $v \in D^W_{p}$.
\end{enumerate}
\end{thm}

To prove the theorem we have to handle the fundamental solution to the
harmonic oscillator diffusion equation. Let us recall that
$$W_t^{\H}f(x)=\int_{\Real^n}W_t^\H(x,y)f(y)\,dy=\int_{\Real^n}\frac{e^{-[\frac{1}{2}|x-y|^2\coth2t+x\cdot y\tanh t]}}{(2\pi\sinh 2t)^{n/2}}f(y)\,dy,$$
see the book by S. Thangavelu \cite[(4.1.2)--(4.1.3)]{Thangavelu}. By applying Stefano Meda's change of parameters
\begin{equation}\label{Meda}
t=\frac{1}{2}\log\frac{1+s}{1-s},\quad t\in(0,\infty),~s\in(0,1),
\end{equation}
we arrive to
\begin{equation}\label{kernel funciones}
W_{t(s)}^\H f(x)=\int_{\Real^n}W_{t(s)}^\H(x,y)f(y)\,dy=\int_{\Real^n}\left(\frac{1-s^2}{4\pi s}\right)^{n/2}e^{-\frac{1}{4}[s|x+y|^2+\frac{1}{s}|x-y|^2]}f(y)\,dy.
\end{equation}

\begin{rem}\label{rem:obvio}
Note that, in \eqref{Meda}, $s\to0^+$ if and only if $t\to0^+$. Moreover, it is clear from \eqref{kernel funciones} that for nonnegative functions $f$ we have $W^\H_{t(s)}f(x)\leq (1-s^2)^{n/2}W_sf(x)$.
\end{rem}

\begin{lem}\label{Lem:dale que va}
Let $f$ be a nonnegative function. Then, for any $x\in\Real^n$ and $0<s<1$,
$$(1-s^2)^{n/2}W_{\frac{9s}{9+25s^2}}f(x)\leq(34/9)^{n/2}e^{\frac{1}{4}s25|x|^2}W^\H_{t(s)}f(x).$$
\end{lem}

\begin{proof}
We just have to compare the kernels in (\ref{kernelclasico}) and \eqref{kernel funciones}. Let us consider two cases.

\texttt{Case 1.} $|y|>4|x|$. Then $|x-y|\leq|x|+|y|\leq\frac{5}{4}|y|$ and $|y|\leq|x-y|+\frac{1}{4}|y|$, so that $\frac{3}{4}|y|\leq|x-y|\leq\frac{5}{4}|y|$. Analogously, $\frac{3}{4}|y|\leq|x+y|\leq\frac{5}{4}|y|$. Consequently, $|x+y|\leq\frac{5}{3}|x-y|$. Hence, in this case,
$$e^{-\frac{1}{4}[s|x+y|^2+\frac{1}{s}|x-y|^2]}\geq e^{-\frac{1}{4}[\frac{25}{9}s|x-y|^2+\frac{1}{s}|x-y|^2]}=e^{-\left(\frac{25s^2+9}{9s}\right)\frac{|x-y|^2}{4}}.$$

\texttt{Case 2.} $|y|\leq4|x|$. Then $0\leq|x+y|\leq5|x|$ and, since $\frac{1}{s}\leq\frac{1}{s}+\frac{25s}{9}$, we get
$$e^{-\frac{1}{4}[s|x+y|^2+\frac{1}{s}|x-y|^2]}\geq e^{-\frac{1}{4}s25|x|^2}e^{-\frac{1}{s}\frac{|x-y|^2}{4}}\geq e^{-\frac{1}{4}s25|x|^2}e^{-\left(\frac{25s^2+9}{9s}\right)\frac{|x-y|^2}{4}}.$$
The result then follows from the two estimates above and by noticing that for every $0<s<1$ we have $\left(\frac{9s}{9+25s^2}\right)^{-n/2}\leq(34/9)^{n/2}s^{-n/2}$.
\end{proof}

\begin{proof}[Proof of Theorem \ref{Thm:funciones}]
By taking into account Remark \ref{rem:obvio} and the change of parameters \eqref{Meda},
\begin{align*}
    W^\H_{t(s)}f(x) &\leq e^{nt(s)}W^\H_{t(s)}|f|(x) = W^{(\H-n)}_{t(s)}|f|(x) \leq e^{nt(s)}(1-s^2)^{n/2}W_s|f|(x) \\
     &= \left(\frac{1+s}{1-s}\right)^{n/2}(1-s^2)^{n/2}W_s|f|(x)\leq 2^nW_s|f|(x).
\end{align*}
This last chain of inequalities and Theorem \ref{Thm:clasico} give $(vi)\Longrightarrow(i)$. The implications $(i)\Longrightarrow(ii)\Longrightarrow(iii)$ are obvious, just notice that the set of continuous functions $\psi$ with compact support are dense in $L^p(\Real^n,v(x)\,dx)$
and it is also well-known that $\lim_{t\to0^+}W_t^{(\H-n)}\psi(x)=\lim_{t\to0^+}e^{tn}W_t^{\H}\psi(x)=\psi(x)$, see \cite[p.~85]{Thangavelu}.

$(iii)\Longrightarrow(iv)$. It is enough to consider $f$ nonnegative. Let $x$ be a point such that $W^{(\H-n)}_Rf(x)<\infty$. By Lemma \ref{Lem:dale que va} there exists $0<s_R<1$ (given by $R=\frac{1}{2}\log\frac{1+s_R}{1-s_R}$) such that
$$W_{\frac{9s_R}{9+25s_R^2}}f(x)<\infty.$$
For $0<s<1$, let $s^\ast:=\frac{9s}{9+25s^2}$. Using Remark \ref{rem:obvio} and Lemma \ref{Lem:dale que va},
\begin{equation}
\begin{aligned}\label{26}
    W_{t(s^\ast)}^{(\H-n)}f(x) &= e^{nt(s^*)}W_{t(s^*)}^\H f(x) \leq \left(\frac{1+s^\ast}{1-s^\ast}\right)^{n/2}\left(1-(s^*)^2\right)^{n/2}W_{s^*}f(x) \\
     &= (1+s^*)^n\left(1-s^2\right)^{-n/2}\left(1-s^2\right)^{n/2}W_{s^*}f(x)  \\
     &\leq 2^n\left(1-s^2\right)^{-n/2} (34/9)^{n/2}e^{\frac{1}{4}s25|x|^2}W_{t(s)}^\H f(x).
\end{aligned}
\end{equation}
As the limit in $(iii)$ exists, we have $\lim_{t\to0^+}W_t^{(\H-n)}f(x)=\lim_{t\to0^+}e^{tn}W_t^\H f(x)$. Hence the chain of inequalities \eqref{26} implies that $\lim_{t\to0^+}W_tf(x)$ exists. Applying Theorem \ref{Thm:clasico}$(c)\Longrightarrow(d)$ we obtain that there exists $0<R<\infty$ such that $W^*_Rf(x)<\infty $ a.e. $x\in\Real^n$. Using again Remark \ref{rem:obvio} we get $(iv)$.

By considering the computations above, $(iv)\Longrightarrow(v)$ is obvious.

Let us finally prove $(v)\Longrightarrow(vi)$. Clearly it is enough to prove that $(iv)$ for the operator $T_R^\ast$ implies $(vi)$. The chain of inequalities in \eqref{26} ensures that there exists a certain $S_R$ such that $\sup_{0<s<S_R}W_sf(x)<\infty$ a.e. $x\in\Real^n$. Therefore we can apply again Theorem \ref{Thm:clasico}$(d)\Longrightarrow(e)$ to have $(vi)$.
\end{proof}

\section{The Ornstein--Uhlenbeck diffusion equation}

In this section we let $\O=-\Delta+2x\cdot\nabla$, the Ornstein--Uhlenbeck operator in 
$\Real^n$, and denote by $W_t^{\O}f(x)\equiv e^{-t\O}f(x)$ the solution to the initial value problem
$$\left\{
  \begin{array}{ll}
    u_t=-\O u, & \hbox{in}~\Real^n,~t>0, \\
    u(x,0)=f(x), & \hbox{on}~\Real^n.
  \end{array}
\right.$$
We denote by $d\gamma(x)=\pi^{-n/2}e^{-|x|^2}dx$, the Gaussian measure on $\Real^n$. The heat kernel for the Ornstein--Uhlenbeck operator is given by Mehler's formula, 
see \cite{Thangavelu}. Instead of working with the kernel we will take advantage of the transference method that was systematically
developed in \cite{Abu-Falahah-Torrea}. The analysis in \cite{Abu-Falahah-Torrea} is carried out by using the isometry $U:L^2(\Real^n,d\gamma(x))\to L^2(\Real^n,dx)$,
defined by $f(x)\longmapsto Uf(x)=\pi^{-n/4}e^{-\frac{|x|^2}{2}}f(x)$, 
see \cite[pp.~414--415,~2.44]{Kamke} and \cite[Lemma~3.1]{Abu-Falahah-Torrea}. It is shown in \cite[Proposition~3.3]{Abu-Falahah-Torrea} that
\begin{equation}\label{transferencia}
U^{-1}W_t^{(\H-n)}Uf(x)=W_t^\O f(x),\quad\hbox{for any polynomial}~f~\hbox{on}~\Real^n.
\end{equation}

\begin{thm}\label{Thm:polinomios}
Let $v$ be a weight in $\mathbb{R}^n$ and  $1\le  p< \infty$. Given $0< R < \infty$, consider the operator
$$\O_R^\ast f(x) = \sup_{t<R}|W_t^{\mathcal{O}}f(x)|,$$
for $f\in L^1_{\mathrm{loc}}(\Real^n)$.
The following statements are equivalent:
\begin{enumerate}[(1)]
    \item There exists  $0< R < \infty$ and a weight $w$ such that the operator
        $$f \longmapsto \O^*_R f$$
        is bounded from $L^p(\mathbb{R}^n, v(x)d\gamma(x))$ into $L^p(\mathbb{R}^n, w(x)d\gamma(x))$ for $p>1$. In the case $p=1$, from $L^1(\mathbb{R}^n,v(x)d\gamma(x))$ into weak-$L^1(\mathbb{R}^n,w(x)d\gamma(x))$.
    \item There exists $0< R < \infty$ and a weight $w$ such that the operator
        $$f \longmapsto \O^*_R f$$
        is bounded from $L^p(\mathbb{R}^n, v(x)d\gamma(x))$  into weak-$L^p(\mathbb{R}^n, w(x)d\gamma(x))$.
    \item There exists $0<R<\infty$ such that $W^{\O}_Rf(x)<\infty $ for almost every $x\in\Real^n$ and the  limit
        $$\lim_{t \rightarrow 0^+} W^{\O}_tf(x)$$
        exists a.e. $x\in\Real^n$, for all $f \in L^p(\mathbb{R}^n, v(x)d\gamma(x))$.
    \item There exists  $0< R < \infty$ such that
        $$\O^*_R f(x) < \infty,$$
        a.e. $x$, for all $ f\in L^p(\mathbb{R}^n, v(x)d\gamma(x))$.
    \item The weight $v \in D^W_{p}$.
\end{enumerate}
\end{thm}

\begin{proof}
The implications \textit{(1)}$\Longrightarrow$\textit{(2)}$\Longrightarrow$\textit{(3)} are obvious.

Let us prove \textit{(3)}$\Longrightarrow$\textit{(4)}. Assume that $f$ is nonnegative. By \eqref{transferencia}, $\lim_{t\to0^+}W_t^{(\H-n)}Uf(x)$ exists for a.e. $x$, for all $f\in L^p(\mathbb{R}^n, v(x)d\gamma(x))$. This is equivalent to saying that $\lim_{t\to0^+}W_t^{(\H-n)}g(x)$ exists a.e for every function $g\in L^p(\mathbb{R}^n,v(x)e^{|x|^2p(-\frac1{p}+\frac12)}dx)$. Moreover, $W_R^{(\mathcal{H}-n)}Uf(x)<\infty$ for a certain $R$. By Theorem \ref{Thm:funciones}$(iii)\Longrightarrow(iv)$ we know that this implies the existence of some $R$ such that the maximal operator $\sup_{t<R}|W_t^{(\H-n)}g(x)|$ is finite almost everywhere, for any function $g\in L^p(\mathbb{R}^n,v(x)e^{|x|^2p(-\frac1{p}+\frac12)}dx)$. Thus the maximal operator $\sup_{t<R}|W_t^{\O}f(x)|$ is finite almost everywhere.

To check \textit{(4)}$\Longrightarrow$\textit{(5)} let us first observe that, by proceeding as above, we see that Theorem \ref{Thm:funciones}$(iv)\Longrightarrow(vi)$ gives that the weight $v(x)e^{|x|^2p(-\frac1{p}+\frac12)}$ belongs to the class $D_p^W$. That is, there exists $M>0$ such that
$$\int_{\Real^n}\Big[ e^{-M|x|^2}\left(v(x)e^{|x|^2p(-\frac1{p}+\frac12)}\right)^{-1/p}\Big]^{p'}\,dx<\infty,$$
when $p>1$, and the corresponding $L^\infty$ bound for $p=1$. It is clear from here that the value of $M$ (which corresponds to $1/t$ in the classical heat kernel) can be chosen as large as we want. In particular, we can take any $M>-\left(-\frac{1}{p}+\frac{1}{2}\right)$.  Then we see that the weight $v$ satisfies Definition \ref{Def:pesos}.

\textit{(5)}$\Longrightarrow$\textit{(1)}. If the weight $v$ satisfies the condition in Definition \ref{Def:pesos} for some $M>0$ then
$$\int_{\Real^n}\Big[e^{-M|x|^2}(v(x))^{-1/p}\Big]^{p'}\,dx<\infty.$$
By choosing $M>\left(-\frac1{p}+\frac12\right)$ we have that
$$\int_{\Real^n}\Big[e^{-\left(M-\left(-\frac1p+\frac12\right)\right)|x|^2}\left(v(x)e^{|x|^2p \left(-\frac{1}{p}+\frac12\right)}\right)^{-1/p}\Big]^{p'}\,dx<\infty.$$
In other words, the weight $v(x)e^{|x|^2p\left(-\frac1{p}+\frac12\right)}$ satisfies Definition \ref{Def:pesos} and therefore $(i)$ of Theorem \ref{Thm:funciones} holds for the maximal operator $\tilde{T}^\ast_R$. Proceeding as in the proof of \textit{(3)}$\Longrightarrow$\textit{(4)} we see that this implies the boundedness of the operator $\O_R^\ast f(x)$ from $L^p(\Real^n,v(x)d\gamma(x))$ into $L^p(\Real^n,w(x)d\gamma(x))$ for some weight $w$. The case $p=1$ follows analogously.
\end{proof}



\end{document}